\documentclass[12pt]{amsart}
\usepackage[english]{babel}
\usepackage{amsfonts,amssymb,latexsym,amscd}

\usepackage[all]{xy}
\usepackage{color}

\usepackage{hyperref}

\theoremstyle{plain}
\newtheorem{theorem}{Theorem}
\newtheorem{lemma}{Lemma}
\newtheorem{proposition}{Proposition}

\theoremstyle{definition}
\newtheorem{definition}{Definition}
\theoremstyle{remark}
\newtheorem{remark}{Remark}

\begin{document}
\title[Yu.~M.~Smirnov's General Equivariant Shape Theory]
{Yu.~M.~Smirnov's General Equivariant\\ Shape Theory}

\author{P. S. Gevorgyan}

\address{Moscow Pedagogical State University, Russia}

\email{pgev@yandex.ru}

\begin{abstract}
A  general equivariant shape theory  for arbitrary $G$-spaces
in the case of a compact group $G$  is constructed by using the method of pseudometrics suggested by Yu.~M.~Smirnov as early as in 1985 at the fifth Tiraspol symposium on general topology and its applications~\cite{smirn85}.

\end{abstract}


\maketitle

\section{Introduction}

Equivariant shape theory was first constructed by Yu.~M.~Smirnov
for $p$-paracompact spaces in the case of a compact acting group $G$ \cite{smirn85-1}.
In 1985, at the fifth Tiraspol symposium on general topology and its applications, 
Smirnov suggested a new method for constructing equivariant shape theory for arbitrary
$G$-spaces in the case of a compact group $G$, which developed 
his ideas proposed at a topological conference in Keszthely and 
at a conference on geometric topology in Warsaw. This method is based on 
the application of all invariant
continuous pseudometrics on a given $G$-space. As Smirnov mentioned
 in \cite{smirn85}, the key role in the method is played by the 
Arens--Ellis--Gevorgyan theorem on a closed 
isometric equivariant embedding of completely regular Hausdorff (metrizable) $G$-spaces 
in locally convex (normed) linear $G$-spaces~\cite{gev85}.

Unfortunately, a detailed exposition of general equivariant shape theory based on pseudometrics
has never been published,
although this theory was constructed by a completely different method in~\cite{am87}.

In this paper, we fill this lacuna and construct  Smirnov's  general equivariant shape theory 
for arbitrary $G$-spaces in the case of a compact group~$G$.

\section{Basic Definitions and Auxiliary Results}

Throughout the paper, we assume $G$ to be a \textit{compact} topological group.

Given a topological space $X$, a continuous map $\pi :G \times X \to X$ 
(we use the notation $\pi(g, x) = gx$) 
is called an \textit{action of the group} $G$ on
$X$ if, for any $g, h \in G$ and $x\in X$, 
$$g(hx)=(gh)x\quad\text{and}\quad ex=x ,$$
where $e \in G$ is the identity element of the group $G$.
A triple $(X,G,\pi)$, or a topological space $X$ with an action of a group $G$,  
is called a \textit{$G$-space}.

A metric $\rho$ on a metrizable $G$-space $(X,\rho)$ is said to be \textit{invariant} 
if $\rho(gx, gy)=\rho(x,y)$. In this case, $X$ is called an \textit{invariantly metrizable} 
$G$-space.

Any metrizable $G$-space is invariantly metrizable~\cite{p60}.

Let $X$ be a linear topological space. An action $\pi :G \times X \to X$  is said to 
be \textit{linear} if 
$$g(\lambda x+\mu y)=\lambda gx+\mu gy$$
for all $g\in G$, $x, y\in X$, and $\lambda, \mu\in \mathbb{R}$.

A subset $A$ in a $G$-space $X$ is said to be \textit{invariant} if
$ga\in A$ for any $g\in G$ and $a\in A$. Obviously, any invariant subset
$A$ itself is a $G$-space.

Let $X$ and $Y$ be $G$-spaces. A continuous map
$f: X\to Y$ is called an \textit{equivariant map}, or a \textit{$G$-map},
if $f(gx)=gf(x)$ for any $g\in G$ and $x\in X$. An equivariant homeomorphism
$f: X\to Y$ is called an \textit{equimorphism} of the $G$-spaces $X$ and $Y$.

We say that equivariant maps $f_0, f_1 :X\to Y$ are \textit{$G$-homotopic}, 
or \textit{equivariantly homotopic},
and write $f_0\simeq_G f_1$ if there exists a $G$-map $F:X\times I \to Y$ such that
$F(x,0)=f_0(x)$ and $F(x,1)=f_1(x)$ for all $x\in X$ (it is assumed that 
the group $G$ acts trivially  on the interval $I=[0,1]$,
and its action on $X\times I$ is defined coordinatewise by $g(x,t)=(gx,t)$).
In this case, $F$ is called a \textit{$G$-homotopy} between $f_0$ and $f_1$.

We denote the homotopy class of an equivariant map $f:X\to Y$ by $[f]$.

Let $Z$ be a $G$-space, and let $Y\subseteq Z$ be its invariant subset.
A $G$-map $r:Z \to Y$ is called an \textit{equivariant retraction} if $r|_Y=id_Y$.

A $G$-space $Y$ is an \textit{equivariant absolute neighborhood retract}, or a 
$G$-ANR-\textit{space} (an \textit{equivariant absolute retract}, or a $G$-AR-\textit{space}), 
if $Y$ is metrizable and, for any closed
equivariant embedding $Y\subset Z$ in a metrizable $G$-space $Z$, there exists an 
invariant neighborhood $U$ of the invariant subset $Y$ and an equivariant retraction $r:U \to Y$
(there exists an equivariant retraction $r:Z \to Y$).

A $G$-space $Y$ is an \textit{equivariant absolute neighborhood extensor}, or 
a $G$-ANE-\textit{space} (an \textit{equivariant absolute extensor}, or 
a $G$-AE-\textit{space}), if, given any metrizable
$G$-space $X$ and its closed invariant subset $A\subset X$, each equivariant map
$f:A\to Y$ can be extended to an equivariant map $F:U\to Y$ ($F:X\to Y$), where $U$ 
is some invariant neighborhood of the subset $A$.

A metrizable $G$-space $Y$ is a $G$-ANR ($G$-AR) space if and only if it is a 
$G$-ANE ($G$-AE) space \cite{am87}.

We denote the homotopy category of all $G$-spaces by H-$G$-TOP.
In  this category, objects are $G$-spaces and morphisms are  
homotopy classes of equivariant maps.
Let H-$G$-ANR be the complete subcategory of $G$-ANR-spaces in the category H-$G$-TOP.

\begin{theorem}[\cite{am87}]\label{th2}
Suppose that $X$ is a metrizable $G$-space, $A\subset X$ is a closed invariant subset,
$Y$ is a $G$-ANR-space, and $f,g:X\to Y$ are equivariant maps such that 
 their restrictions $f|_A$ and $g|_A$ 
are connected by an equivariant homotopy $H:A\times I\to Y$.
Then there exists an equivariant neighborhood $V\supset A$ and an equivariant homotopy 
$\tilde{H}:V\times I\to Y$
between the restrictions $f|_V$ and $g|_V$ for which $\tilde{H}|_{A\times I}=H$.
\end{theorem}

\begin{proof}
Consider the closed invariant subset $C=X\times \{0\}\cup A\times I \cup X\times \{1\}$
of the $G$-space $X\times I$ and the equivariant map $h:C \to Y$ defined by 
$$h(x,0)=f(x), \quad h(x,1)=g(x), \quad h(a,t)=H(a,t),$$
where $x\in X$, $a\in A$, and $t\in I$.

Let $\tilde{h}:U\to Y$ be an equivariant extension of the map $h$
to some invariant neighborhood $U$ of the closed invariant set $C$. Since $I$ is compact,
it follows that  $A$ has an invariant neighborhood $V$ in $X$ for which 
$V\times I\subset U$. It remains to note that $\tilde{H}=\tilde{h}|_{V\times I}$ 
is the required equivariant homotopy
between $f|_V$ and $g|_V$.
\end{proof}

\begin{definition}[{Morita \cite{morita}}]\label{defass}
An inverse system $\{X_\lambda , [p_{\lambda \lambda '}], \Lambda\}$ in the category
H-$G$-ANR is said to be \textit{associated}
with a $G$-space $X$ if 

$1^\circ$ \ for each $\lambda \in \Lambda$,  there exists an equivariant map
$p_\lambda : X \to X_\lambda $ such that
$[p_{\lambda \lambda'}] \circ [p_{\lambda'}]=[p_\lambda]$ 
for any $\lambda, \lambda' \in \Lambda$, $\lambda<\lambda'$;

$2^\circ$ \ for any $G$-ANR-space $Q$ and equivariant map $f:X\to Q$, 
there exists a $\lambda \in \Lambda$ and an equivariant map 
$f_\lambda:X_\lambda \to Q$ such that $[f]=[f_\lambda] \circ [p_\lambda]$;

$3^\circ$ \ for any $\lambda \in \Lambda$ and any equivariant maps
$f_\lambda, g_\lambda :X_\lambda \to Q$ satisfying the condition 
$[f_\lambda] \circ [p_\lambda] = [g_\lambda] \circ [p_\lambda]$,
there exists a $\lambda ' \in \Lambda$, $\lambda \leqslant \lambda '$, for which 
$[f_\lambda] \circ [p_{\lambda \lambda '}] = [g_\lambda] \circ [p_{\lambda \lambda '}]$.
\end{definition}

A detailed exposition of shape theory and equivariant topology is contained in 
books \cite{b72}, \cite{ms82}, 
and \cite{br72}.

\section{The Equivariant Arens--Eells Theorem}

In \cite{A-E}, Arens and Ellis proved a theorem on a closed
uniform (isometric) embedding of  a uniform (metrizable) space
in a locally convex (normed) topological space.
An equivariant generalization of this theorem (see \cite{gev85} and \cite{gev3})
plays the key role in the construction of  Smirnov's general shape theory based on 
pseudometrics.

Below we briefly describe the Arens--Ellis construction of an embedding 
of a metrizable space in a normed linear space.

Let $(X, \rho)$ be a metric space, and let  $M\left(X\right)$ denote
the set of real-valued functions $m(x)$ satisfying the following two conditions:

\begin{enumerate}
\item $m(x)$ vanishes at all points of $X$ except, possibly, at 
finitely many points $x_1,... , x_n$;

\item $\sum\limits_{i=1}^n {m(x_i )} =0$. 
\end{enumerate}

It is easy to see that $M\left( X \right)$ is a real linear space
under pointwise addition of functions and multiplication of a function by a scalar.

By $\lambda x$, where $\lambda \in R$ and $x\in X$, we denote
the function taking the value $\lambda$ at the point
$x$ and vanishing at the other points. Clearly, any function $m\in M\left(
X \right)$ can be represented as 
\[
m=\sum\limits_{i=1}^n {\lambda _i x_i } ,
\]
where $x_i \in X$ and $\sum\limits_{i=1}^n {\lambda _i } =0$.

Given a fixed point $x^0\in X$ in $X$, we set 
\[
B_{x^0} =\left\{ {x-x^0\in M\left( X \right)\;;\quad x\in X, \ x\ne x^0}\right\},
\]
where $x-x^0$ is the function taking the value $1$ at the point $x$, 
$-1$ at the point $x^0$, and vanishing at all other
points.

\begin{lemma}[\cite{A-E}]
The set $B_{x^0}$ is a Hamel basis for the linear space $M\left(X\right)$.
\end{lemma}

The linear space $M\left(X\right)$ is endowed with the norm $s$ defined by 
\begin{equation}\label{eq2}
s\left(m\right)=\inf \left\{{\sum\limits_j{\left| {\mu _j }
\right|\rho \left( {y_j, z_j } \right)} } \right\},
\end{equation}
where the greatest lower bound is over all possible 
decompositions $m=\sum\limits_j{\mu _j \left( {y_j -z_j } \right)}$
of the element $m\in M(X)$.

Consider the map $i:X\to M(X)$ defined by 
\begin{equation}\label{eq5}
i(x)=x-x_0.
\end{equation}

\begin{theorem}[\cite{A-E}]
The map $i$ is a closed isometric embedding of the metric space $X$ in the normed
space $M\left(X\right)$.
\end{theorem}

Now, let $\left(X, G, \pi\right)$ be a metrizable $G$-space.

We define a map $\tilde {\pi }:G\times M\left( {X} \right)\to
M\left( {X} \right)$ by 
\begin{equation}
\label{eq3}
\tilde {\pi }\left( {g,m} \right)=\sum\limits_{i=1}^n {\lambda _i \left(
{gx_i -x^0} \right),}
\end{equation}
where $g\in G$, $x_0\in X$ is a fixed point, $gx_i=\pi(g,x_i)$, 
and $m=\sum\limits_{i=1}^n {\lambda _i \left( {x_i -x^0}
\right)}$ is an arbitrary point in the space $M\left( {X} \right)$ 
decomposed in the basis $B_{x^0} $.

\begin{lemma}[\cite{gev3}]\label{lem2}
The map $\tilde {\pi }$ defined by (\ref{eq3}) 
is a linear action of the group $G$ on the normed space $M\left(X\right)$.
\end{lemma}

\begin{theorem}[\cite{gev3}]\label{th3}
The map $i$ defined by \eqref{eq5} is a closed isometric equivariant
embedding of the metrizable $G$-space $X$ in the linear normed $G$-space $M\left(X\right)$.
\end{theorem}

\begin{theorem}[\cite{gev3}]\label{th1}
Any metrizable $G$-space  $X$ can be  isometrically
equivariantly embedded in the linear normed $G$-space $M\left(X\right)$ as a closed 
subspace.
\end{theorem}

\begin{remark}
Lemma~\ref{lem2} and Theorems~\ref{th3} and~\ref{th1} 
remain valid without the assumption that the group~$G$ is  compact 
provided that  the $G$-space $X$ is  invariantly metrizable.
\end{remark}

Let $X$ and $Y$ be metrizable $G$-spaces, and  let $x_0\in X$ and $y_0\in Y$ be 
fixed points. Suppose that   $f:X\to Y$ is an equivariant map such that $f(x_0)=y_0$. 
Consider the map
$\bar{f}:M(X)\to M(Y)$ defined by 
\begin{equation}\label{eq6}
\bar{f}(m)=\bar{f}\left(\sum\limits_{i=1}^n {\lambda _i \left( {x_i -x^0}
\right)}\right)=\sum\limits_{i=1}^n {\lambda _i \left( {f(x_i) -y^0}
\right)}.
\end{equation}
It is easy to see that the map $\bar{f}$ is a continuous equivariant extension
of the equivariant map $f:X\to Y$.

This implies the following assertion.

\begin{lemma}\label{lem1}
Any equivariant map $f:X\to Y$ can be extended to an equivariant map
$\bar{f}:M(X)\to M(Y)$.
\end{lemma}

\section{Pseudometrics and Inverse Systems}

Let $X$ be a $G$-space.
Consider an invariant continuous pseudometric $\mu$ on $X$.
Setting $x\sim x'$ if and only if $\mu(x,x')=0$, we define an equivalence relation on $X$.
Let $X_{\mu}=X|_\sim$, and let  $[x]_{\mu}$ denote 
the equivalence class of $x\in X$.
Since $x\sim x'$ implies $gx\sim gx'$ for any $x,x'\in X$ and $g\in G$             
(because of the invariance of the pseudometric $\mu$), it follows that 
the quotient space $X_{\mu}$
is a $G$-space with the action $g[x]_{\mu}=[gx]_{\mu}$, and 
$\rho ([x]_{\mu},[x']_{\mu})=\mu(x,x')$
is an invariant metric on $X_{\mu}$. Note also that the quotient map $p_\mu:X\to X_\mu$ defined 
by $p_\mu(x)=[x]_{\mu}$ is continuous and equivariant.

Thus,  the following assertion is valid.

\begin{proposition}
The quotient space $X_\mu$ is an invariantly metrizable $G$-space,
and the quotient map $p_\mu:X\to X_\mu$ is a continuous equivariant map.
\end{proposition}

Let $P$ be the family of all invariant pseudometrics on a $G$-space $X$.
We introduce a natural order on this family by setting 
$\mu\leqslant \mu'$ if $\mu(x,x')\leqslant \mu'(x,x')$ for all $x,x'\in X$.

\begin{proposition}
Invariant pseudometrics $\mu, \mu'\in P$, $\mu\leqslant \mu'$, on a $G$-space $X$ 
generate a natural 
equivariant projection $p_{\mu\mu'}:X_{\mu'}\to X_\mu$ 
not increasing distances and satisfying
the conditions

(i) $p_{\mu\mu'}p_{\mu'}=p_{\mu}$ and (ii) $p_{\mu\mu'}p_{\mu'\mu''}=p_{\mu\mu''}$

\noindent
for any $\mu, \mu', \mu'' \in P$ such that $\mu\leqslant \mu'\leqslant \mu''$.
\end{proposition}

\begin{proof}
The map $p_{\mu\mu'}:X_{\mu'}\to X_\mu$ is given by $p_{\mu\mu'}([x]_{\mu'})=[x]_{\mu}$,
where $[x]_{\mu'}\in X_{\mu'}$ and $[x]_{\mu}\in X_{\mu}$. This map  is well defined, 
because its definition does not depend on the choice of a representative in the class
$[x]_{\mu'}$. Indeed, let $[x]_{\mu'}=[x']_{\mu'}$, that is, $\mu'(x,x')=0$. Then
$\mu(x,x')\leqslant \mu'(x,x')=0$, i.e., $\mu(x,x')=0$. This means that $[x]_{\mu}=[x']_{\mu}$, 
whence $p_{\mu\mu'}([x]_{\mu'})=p_{\mu\mu'}([x']_{\mu'})$.

Conditions (i) and (ii) are easy to verify.
\end{proof}

Thus, $\{X_\mu, p_{\mu\mu'}\}$ is an inverse system of invariantly metrizable $G$-spaces.
By Theorem~\ref{th1}, the $G$-space $X_\mu$ can be  isometrically equivariantly
embedded in the normed $G$-space $M(X_\mu)$ as a closed subspace, 
and, according to Lemma~\ref{lem1}, the equivariant maps
$p_{\mu\mu'}:X_{\mu'}\to X{_\mu}$ can be extended to equivariant maps 
$\bar{p}_{\mu\mu'}:M(X_{\mu'})\to M(X{_\mu})$,
which satisfy the relation $\bar{p}_{\mu\mu'}\bar{p}_{\mu'\mu''}=\bar{p}_{\mu\mu''}$ as well.
Thereby, we obtain the inverse system $\{M(X_\mu), \bar{p}_{\mu\mu'}\}$ of normed $G$-spaces.

Let $\Lambda$ be the set of all pairs $(\mu, U)$, where $\mu\in P$ 
and $U$ is an open invariant neighborhood
of the metrizable $G$-space $X_{\mu}$ in the normed $G$-space $M(X_\mu)$. 
On the set $\Lambda$, we define a 
natural order as follows: given $\lambda=(\mu, U)$ and $\lambda'=(\mu', U')$, we 
set $\lambda\leqslant \lambda'$ if and only if $\mu\leqslant \mu'$ 
and $\bar{p}_{\mu\mu'}(U')\subset U$.
Now, we set $X_\lambda=U$ and consider the equivariant map 
$p_{\lambda\lambda'}:X_{\lambda'}\to X_{\lambda}$
defined by $p_{\lambda\lambda'}=\bar{p}_{\mu\mu'}|_{U'}:U'\to U$. Note 
that if $\mu=\mu'$, then 
$\bar{p}_{\mu\mu'}=id:M(X{_\mu})\to M(X{_\mu})$; 
therefore, $p_{\lambda\lambda'}=i:U'\hookrightarrow U$.

\begin{theorem}\label{th-main}
The inverse system $\{X_\lambda, [p_{\lambda\lambda'}], \Lambda\}$ in the category H-$G$-ANR
is associated with the $G$-space $X$.
\end{theorem}

\begin{proof}
Given $\lambda=(\mu, U)\in \Lambda$, 
consider the equivariant map $p_{\lambda}:X\to X_{\lambda}=U$ defined by 
$p_{\lambda}=ip_{\mu}$, where $p_\mu:X\to X_\mu$ is  the quotient map 
and $i:X_\mu \to U$ is an embedding.
It is easy to see that $p_\lambda=p_{\lambda\lambda'}p_{\lambda'}$ 
if $\lambda \leqslant \lambda'$.

Now, let $f:X\to Q$ be  a continuous equivariant map, where $Q$ is any $G$-ANR-space, 
and let $\rho$ be an invariant metric on the $G$-space $Q$ 
(its existence is ensured by the compactness of the group $G$).
On the $G$-space $X$, we define a continuous invariant pseudometric $\mu$ by 
$\mu(x,x')=\rho(f(x),f(x'))$.

Consider the quotient map $p_\mu:X\to X_\mu$. Note that the map $\varphi:X_\mu \to Q$  given
by $\varphi([x]_\mu)=f(x)$ determines an equimorphism between 
the $G$-spaces $X_\mu$ and $f(X)$, and $\varphi p_\mu=f$.

Since $X_\mu$ is a closed invariant subset of the normed $G$-space $M(X_\mu)$ 
and $Q$ is a $G$-ANR-space,
it follows that  the map $\varphi:X_\mu \to Q$ has an equivariant extension $h: U\to Q$ to  
 some invariant neighborhood $U$ of $X_\mu$ in $M(X_\mu)$. In other words, 
$h|_{X_\mu}=\varphi$, or
$h i=\varphi$.

Let $\lambda=(\mu, U)$. Then $X_\lambda=U$, $p_{\lambda}=ip_{\mu}$, and $h p_\lambda=f$.
Indeed,  we have $h p_\lambda = h ip_{\mu}= \varphi p_{\mu}=f$.

Now, take any $\lambda=(\mu, U)\in \Lambda$ and suppose that
$h_0, h_1:X_\lambda\to Q$ are equivariant maps to a $G$-ANR-space $Q$ for which 
$h_0p_\lambda\simeq_G h_1p_\lambda$. Since $X_\lambda=U$ and $p_{\lambda}=ip_{\mu}$, it follows 
that the equivariant maps $h_0|_{X_\mu}$ and $h_1|_{X_\mu}$ are equivariantly homotopic. 
Thus, by Theorem~\ref{th2}, 
the closed invariant subset $X_\mu$ has an invariant neighborhood $V\subset U$ 
for which $h_0|_V$ and $h_1|_V$ are equivariantly homotopic.
This means that $h_0p_{\lambda\lambda'}\simeq_G h_1p_{\lambda\lambda'}$, where
$\lambda'=(\mu, V)$.

\end{proof}

\section{Equivariant Shape Category}

As is known, shape theory can be constructed for any 
category $\mathcal{T}$ and its complete subcategory $\mathcal{P}$
satisfying the only condition that $\mathcal{P}$
is a dense subcategory in the category $\mathcal{T}$ (see \cite{ms82}). This means that any object
in the category $\mathcal{T}$ has an associated inverse system in the category $\mathcal{P}$.

Theorem~\ref{th-main} asserts that any $G$-space $X$ has an associated inverse system 
in the category H-$G$-ANR; thus,  the following theorem is valid.

\begin{theorem}
The category H-$G$-ANR is a dense subcategory in the category H-$G$-TOP.
\end{theorem}

This theorem makes it possible to construct an equivariant shape theory 
for the class of all $G$-spaces by following the 
general scheme of construction of shape theory for any
category $\mathcal{T}$ and its dense subcategory $\mathcal{P}$ (see~\cite{ms82}).


\end{document}